\documentclass{article}
\usepackage{amsthm}
\usepackage{amsmath}
\usepackage{amssymb}
\usepackage{xparse}
\NewDocumentCommand{\eulerian}{omm}
 {%
  \IfNoValueTF{#1}
   {\genfrac<>{0pt}{}{#2}{#3}}%
   {\mathchoice{\genfrac<>{0pt}{}{#2}{#3}_{\!\!#1}}
               {\genfrac<>{0pt}{}{#2}{#3}_{\mkern-2mu#1}}% \! is -3mu
               {\genfrac<>{0pt}{}{#2}{#3}_{\mkern-2mu#1}}
               {\genfrac<>{0pt}{}{#2}{#3}_{\mkern-2mu#1}}%
   }
 }

\newtheorem{theorem}{Theorem}

\begin{document}
\title{Two new explicit formulas for the Bernoulli Numbers}
\author{Sumit Kumar Jha\\
IIIT-Hyderabad, India\\
Email: kumarjha.sumit@research.iiit.ac.in}
\maketitle
\begin{abstract}
\noindent
In this brief note, we give two explicit formulas for the Bernoulli Numbers in terms of the Stirling numbers of the second kind, and the Eulerian Numbers. To the best of our knowledge, these formulas are new. We also derive two more probably known formulas. 
\end{abstract}
\noindent
There are many known explicit formulas known for the Bernoulli Numbers [1]. We prove the following here.
\begin{theorem}
We have
\begin{equation}
B_{r+1}=\frac{(-1)^r\cdot (r+1)\cdot 2^r }{2^{r+1}-1}\sum_{k=1}^{r}\frac{S(r,k)}{k+1}(-1)^{k}2^{-2k}\frac{(2k-1)!}{(k-1)!},
\end{equation}
and
\begin{equation}
B_{r+1}=\frac{(-1)^r(r+1)}{2^r(2^{r+1}-1)}\binom{2r}{r-1}\sum_{l=1}^{r}(-1)^l \eulerian{r}{r-l}\frac{\binom{r-1}{l-1}}{\binom{2r}{2l-1}},
\end{equation}
\begin{equation}
    (-1)^{r-1}B_{r}=\sum_{k=1}^{r}(-1)^{k}\frac{S(r,k)}{k+1}\cdot (k-1)!,
\end{equation}
and 
\begin{equation}
    (-1)^{r-1}B_{r}=\sum_{l=1}^{r}\eulerian{r}{r-l}\frac{(-1)^l}{l\cdot \binom{r+1}{l}},
\end{equation}
where $B_{r+1}$ is the Bernoulli number, $S(r,k)$ denotes the Stirling number of the second kind, and $\eulerian{r}{r-l}$ represent the Eulerian numbers.
\end{theorem}
\begin{proof}
Our proof of (1) and (2) relies on the following integral representation for the Riemann Zeta function 
\begin{equation}
    \zeta(s)=\frac{1}{\pi(2-2^s)}\int_{0}^{\infty}\frac{x^{-1/2}\textbf{Li}_{s}(-x)}{1+x}\, dx,
\end{equation}
where $\textbf{Li}_{s}(-x)$ is the Polylogarithm function, and is valid for all $s\in \mathbb{C}\setminus\{1\}$. The above is derived using the Ramanujan's Master Theorem in the appendix. The above can also be obtained from formula 3.2.1.6 in the book [3].\par
The integral representation (5) can be used to obtain, for example,
\begin{equation}
\zeta(0) =\frac{1}{\pi}\int_{0}^{\infty}\frac{x^{-1/2}\textbf{Li}_{0}(-x)}{1+x}\, dx=\frac{-1}{2}.
\end{equation}
Also,
\begin{equation}
    \zeta(-r)=\frac{1}{\pi(2-2^{-r})}\int_{0}^{\infty}\frac{x^{-1/2}\textbf{Li}_{-r}(-x)}{1+x}\, dx.
\end{equation}
Now, using the following representation from the note [4]
\begin{equation}
    \textbf{Li}_{-r}(-x)=\sum_{k=1}^rk!S(r,k)\left(\frac{1}{1+x}\right)^{k+1}(-x)^{k},
\end{equation}
which can be easily proved using induction on $r$.\par 
Now, we deduce equation (1) from the following steps
\begin{eqnarray*}
\zeta(-r)=\frac{1}{\pi(2-2^{-r})}\int_{0}^{\infty}\frac{x^{-1/2}}{1+x}\textbf{Li}_{-r}(-x)\, dx\\
=\frac{1}{\pi(2-2^{-r})}\int_{0}^{\infty}\frac{x^{-1/2}}{1+x}\sum_{k=1}^rk!S(r,k)\left(\frac{1}{1+x}\right)^{k+1}(-x)^{k}\, dx\\
=\frac{1}{\pi(2-2^{-r})}\sum_{k=1}^{r}k!\cdot S(r,k)\cdot (-1)^{k}\cdot \beta\left(k+\frac{1}{2},\frac{3}{2}\right)\\
=\frac{1}{\pi(2-2^{-r})}\sum_{k=1}^{r}\frac{1}{k+1}\cdot S(r,k)\cdot (-1)^{k}\cdot \Gamma(k+1/2)\cdot \Gamma(3/2)\\
=\frac{1}{\pi(2-2^{-r})}\sum_{k=1}^{r}\frac{1}{k+1}\cdot S(r,k)\cdot (-1)^{k}\cdot \frac{\Gamma(2k)}{\Gamma(k)}\cdot 2^{1-2k}\cdot\frac{\pi}{2}
\\
=\frac{2^r}{2^{r+1}-1}\sum_{k=1}^{r}\frac{S(r,k)}{k+1}(-1)^{k}2^{-2k}\frac{(2k-1)!}{(k-1)!},
\end{eqnarray*}
and the fact that $\zeta(-r)=(-1)^{r}\frac{B_{r+1}}{r+1}$. Here, $\Gamma(\cdot)$ and $\beta(\cdot, \cdot)$ are the Gamma and Beta function respectively.\par 
Another representation for $\textbf{Li}_{-r}(-x)$ is the following from [5]
\begin{equation}
    \textbf{Li}_{-r}(-x)=\frac{1}{(1+x)^{r+1}}\cdot \sum_{j=0}^{r-1}\eulerian{r}{j}\cdot (-x)^{r-j}.
\end{equation}
Now, to derive equation (2) we follow the steps below
\begin{eqnarray*}
\zeta(-r)=\frac{1}{\pi(2-2^{-r})}\int_{0}^{\infty}\frac{x^{-1/2}}{1+x}\textbf{Li}_{-r}(-x)\, dx\\
=\frac{1}{\pi(2-2^{-r})}\int_{0}^{\infty}\frac{x^{-1/2}}{1+x}\cdot \frac{1}{(1+x)^{r+1}}\cdot \sum_{j=0}^{r-1}\eulerian{r}{j}\cdot (-x)^{r-j}\, dx\\
=\frac{1}{\pi(2-2^{-r})}\sum_{j=0}^{r-1}(-1)^{r-j}\eulerian{r}{j}\cdot \beta\left(r-j+\frac{1}{2},j+\frac{3}{2}\right)\\
=\frac{1}{\pi(2-2^{-r})}\sum_{j=0}^{r-1}(-1)^{r-j}\eulerian{r}{j}\cdot \frac{\Gamma(r-j+1/2)\, \Gamma(j+3/2)}{\Gamma(r+2)}\\
=\frac{1}{\pi(2-2^{-r})}\sum_{l=1}^{r}(-1)^{l}\eulerian{r}{r-l}\cdot \frac{\Gamma(l+1/2)\, \Gamma(r-l+3/2)}{\Gamma(r+2)}\\
=\frac{1}{\pi(2-2^{-r})}\sum_{l=1}^{r}(-1)^{l}\eulerian{r}{r-l}\cdot \frac{2^{1-2l} \sqrt{\pi}\, \Gamma(2l)}{\Gamma(l)\Gamma(r+2)}\cdot \frac{2^{-1-2(r-l)}\sqrt{\pi}\,\Gamma(2(r-l+1)}{\Gamma(r-l+1)}\\
=\frac{1}{2^{r}(2^{r+1}-1)\cdot(r+1)!}\sum_{l=1}^{r}(-1)^l \eulerian{r}{r-l}\frac{(2l-1)!}{(l-1)!}\cdot \frac{(2r-2l+1)!}{(r-l)!}\\
=
\frac{1}{2^{r}(2^{r+1}-1)\cdot(r+1)!}\sum_{l=1}^{r}(-1)^l \eulerian{r}{r-l}\frac{\binom{r-1}{l-1}}{\binom{2r}{2l-1}}\cdot \frac{(2r)!}{(r-1)!},
\end{eqnarray*}
and the fact that $\zeta(-r)=(-1)^{r}\frac{B_{r+1}}{r+1}$.\par 
To prove (3) and (4), we use the following integral representation
\begin{equation}
    \zeta(s+1)=\frac{-1}{s}\int_{0}^{\infty}\frac{\textbf{Li}_{s}(-x)}{(x)(1+x)}\, dx,
\end{equation}
which can be derived from the result in the appendix, and the representations (8) and (9).
\end{proof}
\section*{Appendix}
Following result appears in [2] in a different form.
\begin{theorem}
\normalfont
For all $s\in \mathbb{C}\setminus \{1\}$, and $0<n<1$, we have
\begin{equation}
\int_{0}^{\infty}x^{n-1}\frac{\textbf{Li}_{s}(-x)}{1+x}dx=\frac{\pi}{\sin{n\pi}}(\zeta(s)-\zeta(s,1-n)),
\end{equation}
where $\zeta(s,1-n)$ represents the Hurwitz Zeta function.
\end{theorem}
\begin{proof}
Let
$$
H^{(s)}_{n}=\sum_{k=1}^{n}\frac{1}{k^s}
$$
be the generalized Harmonic number. Then, we have the following generating function from [7]
$$
\frac{\textbf{Li}_{s}(x)}{1-x}
=\sum_{n=1}^{\infty}H^{(s)}_{n}x^{n},
$$
for $|x|<1$. We can also write
\begin{equation}
\frac{\textbf{Li}_{s}(-x)}{(1+x)}=\sum_{n=1}^{\infty}H_{n}^{(s)}(-x)^{n}.
\end{equation}
We have the following explicit form 
\begin{equation}
H^{(s)}_{n}
=\zeta(s)-\zeta(s,n+1).
\end{equation}
Next, we use Ramanujan's Master Theorem (\textsc{RMT}) from [6], which is,
\begin{equation}
\label{two}
    \int_{0}^{\infty} x^{n-1}\{ \phi(0)-x \phi(1)+x^{2} \phi(2)-\cdots\}\, dx=\frac{\pi}{\sin{n \pi}}\phi(-n),
\end{equation}
where the integral is convergent for $0<\textbf{Re}(n)<1$, and after certain conditions are satisfied by $\phi$.
Now, using \textsc{RMT} with equations (12) and (13) gives us required equation (11).
\end{proof}
\section*{References}
\begin{itemize}
\item[1] Gould, Henry W. "Explicit formulas for Bernoulli numbers." The American Mathematical Monthly 79.1 (1972): 44-51.
\item[2] Jha, S. K. (2019, May 24). An integral representation for the Riemann zeta function on positive integer arguments.\\ https://doi.org/10.31219/osf.io/ufwep
\item[3] Handbook of Mellin Transforms by Yu. Brychkov, O. Marichev, N, Savischenko, 2019.
\item[4] Stirling Numbers and Polylogarithms by Steven E. Landsburg. Link: http://www.landsburg.com/query.pdf
\item[5] Weisstein, Eric W. "Eulerian Number." From MathWorld--A Wolfram Web Resource.
\item[6] Amdeberhan, T., Espinosa, O., Gonzalez, I., Harrison, M., Moll, V. H., \& Straub, A. (2012). Ramanujan's master theorem. The Ramanujan Journal, 29(1-3), 103-120.
\item[7.] Weisstein, Eric W. "Polylogarithm." From MathWorld--A Wolfram Web Resource. 
\end{itemize}
\end{document}